\documentclass[12pt]{amsart}
\usepackage[all]{xy}
\usepackage[mathscr]{euscript}
\usepackage{amsmath,amssymb,amsthm}
\usepackage{hyperref}
\hypersetup{colorlinks=true,urlcolor=blue,citecolor=blue,linkcolor=blue}
\usepackage{courier}
\usepackage{amsmath,amscd,amssymb,latexsym,longtable,diagram, picture}
\usepackage{graphicx}
\usepackage{array}
\usepackage{color}
\usepackage{enumerate}
\usepackage{nicefrac}
\usepackage{listings}
\usepackage{latexsym,bm,bbm,mathrsfs}
\usepackage{hyperref,graphicx, enumerate}
\usepackage{epsfig, xcolor, graphicx}
\usepackage[shortlabels]{enumitem}
\usepackage{indentfirst, setspace}
\usepackage{tikz-cd}

\allowdisplaybreaks
\usepackage{caption}
\usepackage{cancel}
\usepackage[normalem]{ulem}
\usetikzlibrary{calc,matrix,arrows,decorations.markings}

\newcommand{\bbz}{\mathbb{Z}}
\newcommand{\bbq}{\mathbb{Q}}
\newcommand{\bbp}{\mathbb{P}}

\newcommand{\End}{\mathrm{End}}
\newcommand{\Aut}{\mathrm{Aut}}
\newcommand{\Gal}{\mathrm{Gal}}

\newcommand{\cA}{\mathcal{A}}

\newcommand{\cS}{\mathcal{S}}
\newcommand{\mat}{\begin{pmatrix}}
\newcommand{\emat}{\end{pmatrix}}
\newcommand{\rank}{\mathrm{rank}}

\newcommand{\Spec}{\mathrm{Spec}}

\newtheorem{theorem}{Theorem}[section]
\newtheorem{proposition}[theorem]{Proposition}
\newtheorem{corollary}[theorem]{Corollary}
\newtheorem{lemma}[theorem]{Lemma}
\newtheorem{rmk}[theorem]{Remark}

\newtheorem{definition}[theorem]{Definition}
\newtheorem{conjecture}[theorem]{Conjecture}

\newtheorem{question}[theorem]{Question}

\begin{document}
 \title[Rank growth of abelian varieties]
{Rank growth of abelian varieties over certain finite Galois extensions}

\author{Seokhyun Choi and Bo-Hae Im}

\address{
Dept. of Mathematical Sciences, KAIST,
291 Daehak-ro, Yuseong-gu,
Daejeon 34141, South Korea
}
\email{bhim@kaist.ac.kr}

\address{
Dept. of Mathematical Sciences, KAIST,
291 Daehak-ro, Yuseong-gu,
Daejeon 34141, South Korea
}
\email{sh021217@kaist.ac.kr}

\date{\today}
\thanks{Bo-Hae Im was  supported by Basic Science Research Program through the National Research Foundation of Korea(NRF) funded by the Korea government(MSIT)(NRF-2023R1A2C1002385, or RS-2023-NR076333).}

\begin{abstract}
    Let $A/K$ be an abelian variety over a number field $K$. We prove that a finite automorphism group $G \subseteq \Aut_K(X)$ of a smooth projective variety $X/K$ such that $X/G \cong \bbp_K^d$ can force the rank growth of $A$ over infinitely many mutually linearly disjoint $G$-extensions $L_i/K$. The proof is based on Hilbert irreducibility and N\'{e}ron specialization. We then combine the theorem with finite group representations to obtain explicit lower bounds for rank growth. As applications, we obtain rank growth results for Jacobian varieties and construct explicit examples. We further prove arbitrarily large rank growth of abelian varieties over symmetric extensions. Finally, we study the connection with infinite rank conjectures.
\end{abstract}

\maketitle

\section{Introduction}

Abelian varieties, as higher-dimensional generalizations of elliptic curves, lie at the heart of modern number theory and algebraic geometry. Their rich internal structure provides a vital bridge between geometry and arithmetic. A foundational result governing abelian varieties over 
number fields is the Mordell--Weil theorem, which establishes the fundamental structure of the group of rational points on these varieties.

The Mordell--Weil theorem asserts that if $A$ is an abelian variety defined over a number field $K$, then $A(K)$ is a finitely generated abelian group. This implies that $ A(K)$ can be decomposed as follows:
\begin{equation*}
    A(K) = A(K)_{tor} \oplus \mathbb{Z}^r,
\end{equation*}
where $A(K)_{tor}$ is the torsion subgroup of $A(K)$ consisting of all points of finite order, and
$r$ is the rank of
$ A(K)$. The study of torsion subgroups of abelian varieties has been an important focus in the field. For example, Mazur \cite{Maz77, Maz78}  completely classified the torsion subgroups of elliptic curves over
$\mathbb{Q}$, and Merel \cite{Mer96} established the uniform boundedness conjecture for torsion subgroups of elliptic curves over number fields.

On the other hand, the rank $r$ of an abelian variety offers deep insights into its structure and complexity relative to the field $K$. However, ranks are highly challenging to analyze, as doing so requires the development of diverse and groundbreaking approaches, and many open conjectures remain. One significant problem is understanding how much the rank can increase when the base field is extended. Although some partial results are known, this area still holds many unanswered questions.

We begin by reviewing relevant known results. Throughout this introduction, let $A/K$ be an abelian variety over a number field $K$. For a finite extension $L/K$, if $L/K$ is Galois with $\Gal(L/K)=G$, then we call $L/K$ a $G$-extension. 

The study of rank growth of abelian varieties over finite extensions originates with the work of Frey and Jarden \cite{FJ74}. They proved that for every positive integer $n$, there exists a finite extension $L$ over $K$ such that 
\[\rank(A(L)) \geq \rank(A(K)) + n.\] 
Thus, the rank of abelian varieties can grow arbitrarily large over finite extensions.

Bruin and Najman \cite{BN16} investigated restrictions on possible rank growth imposed by the Galois group of a finite Galois extension. For example, if $\lvert G \rvert$ is odd and $p$ is the smallest prime factor of $\lvert G \rvert$, then either 
\[\rank(A(L)) = \rank(A(K)) \quad \text{or} \quad \rank(A(L)) \geq \rank(A(K)) + (p-1).\] 
They obtained several similar representation-theoretic restrictions on possible rank growth. However, these results apply only when rank growth occurs; they do not establish the existence of field extensions over which the rank actually increases.

On the other hand, rank growth is not always guaranteed. Mazur and Rubin \cite{MR18} proved that, under suitable hypotheses on $A$ and $K$, there exist infinitely many finite Galois extensions $L$ over $K$ for which 
\[\rank(A(L)) = \rank(A(K)).\] 

Conversely, several authors have identified settings where rank strictly increases over infinitely many finite extensions. For an elliptic curve $E/K$,  Fearnley, Kisilevsky, and Kuwata \cite{FKK12} (for $n=3$) and later Kashyap \cite{Kas13}  (for $n \in \{3,4,6\}$) proved that if $K$ contains a primitive $n$-th root of unity, there exist infinitely many $\bbz/n\bbz$-extensions $L/K$ such that $$\rank(E(L)) > \rank(E(K)).$$
Over the rationals, David, Fearnley, and Kisilevsky \cite{DFK07} conjectured analogous rank growth over infinitely many $\bbz/p\bbz$-extensions $L/\bbq$ for $p \in \{3,5\}$, while predicting it fails for primes $p>5$.

For non-abelian extensions, Lemke Oliver and Thorne \cite{OT21} proved that for an elliptic curve $E/\bbq$ and an integer $d \geq 2$, there exist infinitely many degree $d$ extensions $L/\bbq$ with Galois closures $\cS_d$ such that 
\[\rank (E(L)) > \rank (E(\bbq)).\] 
In this direction, the second author and K\"{o}nig \cite[Question~1(c)]{IK23} asked whether the same conclusion holds for all elliptic curves over arbitrary number fields. 

Beyond elliptic curves, the second author and Wallace \cite{IW18} proved that for a Jacobian variety $J/K$ of a smooth projective curve satisfying suitable hypotheses, there exist infinitely many $\bbz/p\bbz$-extensions $L/K$ such that 
\[\rank(J(L))>\rank(J(K)).\]
While the broader literature on the rank growth of abelian varieties is extensive, we focus here only on those results most closely related to the present paper.

Our main theorem gives a general criterion for rank growth of abelian varieties over finite Galois extensions. Its proof is based on two classical theorems, Hilbert irreducibility and N\'{e}ron specialization, and similar ideas already appear in the work of Petersen~\cite{Pet06}, Im and Larsen \cite{IL10}, and Suresh \cite{Sur23}. Our contribution is to combine these ideas into a single theorem that works uniformly over infinitely many mutually linearly disjoint $G$-extensions, and at the same time, controls the rank growth over all of their intermediate fields.

\begin{theorem}\label{main_theorem}
    Let $A/K$ be an abelian variety over a number field $K$. Let $X/K$ be a smooth projective variety with a finite subgroup $G \subseteq \Aut_K(X)$ such that $X/G \cong \bbp^d_K$. Define 
    \[M := \mathrm{Mor}_K(X,A)/A(K),\]
    where we identify $A(K)$ with constant morphisms $X \rightarrow A$ over $K$.
    Then there exist infinitely many mutually linearly disjoint $G$-extensions $L_i/K$ such that 
    \[M \hookrightarrow A(L_i)/A(K)\]
    as $G$-modules and hence 
    \[\rank(A(L_i)) \geq \rank(A(K)) + \mathrm{rank}(M).\]
\end{theorem}

As an immediate consequence of Theorem~\ref{main_theorem}, we obtain the following corollary.

\begin{corollary}\label{main_theorem_corollary}
     Keep the hypotheses and notation of Theorem~\ref{main_theorem}. Assume further that there exists a non-constant morphism $f:X \rightarrow A$ over $K$.
     Then there exist infinitely many mutually linearly disjoint $G$-extensions $L_i/K$ such that 
     \[\rank(A(L_i)) > \rank(A(K)).\]
\end{corollary}

The conclusion in Theorem~\ref{main_theorem} also controls the rank growth of $A$ over all intermediate fields of $L_i/K$. This is one of the main features of our formulation of Theorem~\ref{main_theorem}.

\begin{corollary}\label{main_theorem_fixed_field_corollary}
    Keep the hypotheses and notation of Theorem~\ref{main_theorem}. Let $H \leq G$ and define 
    \[M_H := \mathrm{Mor}_K(X/H,A)/A(K).\]
Then there exist infinitely many mutually linearly disjoint $G$-extensions $L_i/K$ such that 
    \[\rank(A(L_i^H)) \geq \rank(A(K)) + \mathrm{rank}(M_H).\]
\end{corollary}

Sections~\ref{Explicit_lower_bound} to \ref{section_A4} develop the geometric and representation-theoretic consequences of this framework. In Section~\ref{Explicit_lower_bound}, we combine Theorem~\ref{main_theorem} with the representation theory of finite groups to derive explicit lower bounds for rank growth over $G$-extensions. In Section~\ref{section_Jacobian}, we apply Theorem~\ref{main_theorem} to Jacobian varieties, obtaining applications to hyperelliptic curves, modular curves, and Fermat curves. In Section~\ref{section_A4}, we construct explicit examples satisfying the hypotheses of Theorem~\ref{main_theorem}, including a family of elliptic curves whose rank grows by at least $3$ over $\cA_4$-extensions. This example was the original motivation for the present paper. 

In Section~\ref{section_Sn}, we derive two general consequences for symmetric extensions. Frey and Jarden \cite{FJ74} proved that the rank of an abelian variety can grow arbitrarily large over arbitrary finite extensions. We prove that symmetric extensions already suffice. More precisely, we obtain the following theorem.

\begin{theorem}\label{main_theorem_S_n_extensions}
    Let $A/K$ be an abelian variety over a number field $K$. Then there exists a positive integer $N$ such that for every $n \geq N$, there exist infinitely many mutually linearly disjoint $\cS_{n+1}$-extensions $L_i/K$ such that 
    \[\rank(A(L_i)) \geq \rank(A(K)) + n.\]
    If $A$ is an elliptic curve $E$, then one may take $N=1$. 
\end{theorem}

We then pass to the corresponding degree-$n$ extensions. Combined with Corollary~\ref{main_theorem_fixed_field_corollary}, we obtain the following theorem.

\begin{theorem}\label{main_theorem_degree_n_extensions}
    Let $A/K$ be an abelian variety over a number field $K$. Then there exists a positive integer $N$ such that for every $n \geq N$, there exist infinitely many degree-$n$ extensions $F_i/K$ with Galois closure $\cS_n$ such that 
    \[\rank(A(F_i)) > \rank(A(K)).\]
    If $A$ is an elliptic curve $E$, then one may take $N=2$.
\end{theorem}

For elliptic curves, Theorem~\ref{main_theorem_degree_n_extensions} holds in every degree $n \geq 2$. It therefore gives an unconditional affirmative answer to \cite[Question~1(c)]{IK23} for symmetric groups. It also extends the existence result of Lemke Oliver and Thorne~\cite{OT21} from elliptic curves over~$\mathbb Q$ to arbitrary abelian varieties over number fields, in all sufficiently large degrees.

Finally in Section~\ref{section_infinite_rank}, we apply Theorem~\ref{main_theorem} to two infinite rank conjectures, namely the Frey--Jarden conjecture and Larsen's conjecture. By applying Theorem~\ref{main_theorem} to these conjectures, we obtain geometric criteria for these conjectures.

Throughout this paper, we follow the notation in \cite{Sil86}. In particular, if $V$ is a smooth projective variety over a number field $K$, then we often identify $V$ with the set $V(\overline{K})$ of $\overline{K}$-rational points on $V$, where $\overline{K}$ is a fixed algebraic closure of $K$. We assume throughout that all varieties are geometrically integral. 

We use $C_n$ to denote the cyclic group of order~$n$, $\cS_n$ for the symmetric group of degree~$n$, $\cA_n$ for the alternating group of degree $n$, $\mathrm{PSL}_2(\bbz/N\bbz)$ for the projective special linear group of $\bbz/N\bbz$. We denote by $\zeta_n$ a primitive $n$th root of unity.

\section{Proof of Theorem~\ref{main_theorem}}\label{section_main_theorem}

In this section, we prove Theorem~\ref{main_theorem}. We fix an abelian variety $A/K$ over a number field $K$, a smooth projective variety $X/K$ with a finite subgroup $G \subseteq \Aut_K(X)$ such that $X/G \cong \bbp^d_K$, and define 
\[M := \mathrm{Mor}_K(X,A)/A(K).\] 

We begin with the following two lemmas. Let $q: X \rightarrow X/G \cong \bbp^d_K$ be the quotient morphism and set $\lvert G \rvert = m$. 

\begin{lemma}[Hilbert irreducibility theorem]\label{Hilbert}
    There exists a thin subset $\Omega_0$ of $\bbp^d(K)$ such that for every $t \notin \Omega_0$, the fiber $q^{-1}(t)$ consists of exactly $m$ number of $\overline{K}$-rational points. Moreover, for any $x \in q^{-1}(t)$, $x$ generates a $G$-extension $L_t$ over $K$. 
\end{lemma}
\begin{proof}
    It follows from applying \cite[p.123, Proposition 2]{Ser97}.
\end{proof}

Let $F=K(X/G)$, $E=K(X)$ be function fields. Then $E/F$ is a Galois extension with Galois group $G$. We show that every $E$-rational point of $A$ corresponds uniquely to a morphism $X \rightarrow A$ over $K$. 

Given a morphism $f:X \rightarrow A$ over $K$, by composing with the generic point $\eta : \Spec\: E \rightarrow X$ of $X$, we obtain an $E$-rational point $f \circ \eta : \Spec\: E \rightarrow A$ of $A$. Conversely, suppose we are given an $E$-rational point $P:\Spec\: E \rightarrow A$ of $A$. Then we obtain a rational map $f:X \rightarrow A$ over $K$. By \cite[Theorem~3.1]{Mil86}, $f$ extends uniquely to a morphism $f:X \rightarrow A$ over $K$. 

Therefore, we will freely identify morphisms $X \rightarrow A$ over $K$ with $E$-rational points of $A$. Under this identification, we have 
\[M \cong A(E)/A(K).\]
In particular, by the Lang-N\'{e}ron theorem \cite[Theorem~III.6.1]{Sil94}, $M$ is finitely generated.

Fix $t \notin \Omega_0$, choose $x \in q^{-1}(t)$, and let $L_t/K$ be a $G$-extension generated by $x$. Define the specialization homomorphism 
\[\phi_t : A(E) \rightarrow A(L_t)\quad \text{ by } \quad 
\phi_t(P) = f(x),\] 
for a morphism $f:X \rightarrow A$ over $K$ corresponding to an $E$-rational point $P$ of $A$.

\begin{lemma}[N\'{e}ron specialization theorem]\label{Neron}
    There exists a thin subset $\Omega \supseteq \Omega_0$ of~$\bbp^d(K)$ such that for every $t \notin \Omega$, the specialization homomorphism 
    \[\phi_t : A(E) \rightarrow A(L_t)\]
    is injective.
\end{lemma}
\begin{proof}
    Let $A_E$ be the base change of $A$ to $E$, and let $B := \mathrm{Res}_{E/F}(A_E)$ be the Weil restriction of $A_E$ from $E$ to $F$. Apply \cite[p.152, Theorem]{Ser97} to $B/F$ to obtain a thin set $\Omega \supseteq \Omega_0$ of $\bbp^d(K)$ such that for every $t \notin \Omega$, the specialization homomorphism 
    \[\phi_t : B(F) \rightarrow B_t(K)\]
    is injective. By the properties of the Weil restriction, $B(F) \cong A(E)$, $B_t(K) \cong A(L_t)$, and the specialization homomorphism $\phi_t : B(F) \rightarrow B_t(K)$ coincides with our specialization homomorphism $\phi_t : A(E) \rightarrow A(L_t)$.
\end{proof}

We are now ready to prove Theorem~\ref{main_theorem}.

\begin{proof}[Proof of Theorem~\ref{main_theorem}]
    By Lemmas~\ref{Hilbert} and \ref{Neron}, there exists a thin subset $\Omega$ of $\bbp^d(K)$ such that for every $t \notin \Omega$, the fiber $q^{-1}(t)$ consists of exactly $m$ number of $\overline{K}$-rational points and if we let $L_t/K$ be a $G$-extension generated by any $x \in q^{-1}(t)$, then the specialization homomorphism 
    \[\phi_t : A(E) \rightarrow A(L_t)\]
    is injective.
    
    Fix $t \notin \Omega$, and choose $x \in q^{-1}(t)$. Let $L_t/K$ be a $G$-extension generated by $x$. Define 
    \[F_t : M \xrightarrow{\sim} A(E)/A(K) \xrightarrow{\phi_t \bmod{A(K)}} A(L_t)/A(K).\]
    Since $\phi_t$ is injective, $F_t$ is also injective. We next show that $F_t$ is $G$-equivariant. Recall that $G$ acts on $M$ by 
    \[g \cdot f \bmod{A(K)} = f \circ g^{-1} \bmod{A(K)}.\]
    Let $\iota_t:G \xrightarrow{\sim} \Gal(L_t/K)$ be defined by 
    \[\iota_t(g)(x) = g^{-1}x,\quad g \in G.\] 
    Then $G$ acts on $A(L_t)/A(K)$ via $\iota_t$. Given $g \in G$ and $f \bmod{A(K)} \in M$, we have 
    \begin{align*}
        F_t(g \cdot (f \bmod{A(K)})) &= F_t(f \circ g^{-1} \bmod{A(K)}) \\
        &= f(g^{-1}(x)) \bmod{A(K)} \\
        &= f(\iota_t(g)(x)) \bmod{A(K)} \\
        &= \iota_t(g)(f(x)) \bmod{A(K)} \\
        &= \iota_t(g)(F_t(f \bmod{A(K)})) \\
        &= g \cdot F_t(f \bmod{A(K)}).
    \end{align*}
    Therefore, $F_t$ is an injective $G$-equivariant homomorphism.
    
    This proves that for every $t \notin \Omega$, 
    \[M \hookrightarrow A(L_t)/A(K)\]
    as $G$-modules.
    
    We now use induction to construct infinitely many such mutually linearly disjoint extensions, thereby completing the proof. Suppose we have found mutually linearly disjoint $G$-extensions $L_1,\ldots,L_k$ over $K$ satisfying
    \[M \hookrightarrow A(L_i)/A(K).\]
    Let $L = L_1 \cdots L_k$ be the composite field. By applying the above procedure with $K$ replaced by $L$, we obtain a thin subset $\Omega_L$ of $\bbp^d(L)$. By \cite[p.128, Proposition]{Ser97}, the intersection $\Omega_K := \Omega_L \cap \bbp^d(K)$ is a thin subset of $\bbp^d(K)$. As $\Omega$ and $\Omega_K$ are both thin subsets of $\bbp^d(K)$, their union $\Omega \cup \Omega_K$ is also a thin subset of $\bbp^d(K)$. Take $t \notin \Omega \cup \Omega_K$ and let $L_{k+1} := L_t$ be the $G$-extension over $K$ corresponding to $t$. Since $t \notin \Omega_L$, $L_{k+1}L$ is a $G$-extension over $L$. Therefore, $L_{k+1}$ is linearly disjoint from $L$ over $K$, which implies that $L_1,\ldots,L_k,L_{k+1}$ are mutually linearly disjoint over $K$.
\end{proof}

We conclude this section with a proof of Corollary~\ref{main_theorem_fixed_field_corollary}.

\begin{proof}[Proof of Corollary~\ref{main_theorem_fixed_field_corollary}]
    By Theorem~\ref{main_theorem}, there exist infinitely many mutually linearly disjoint $G$-extensions $L_i/K$ such that 
    \[M \hookrightarrow A(L_i)/A(K).\]
    By taking $H$-invariants, we obtain 
    \[M^H \hookrightarrow (A(L_i)/A(K))^H,\]
    so that
    \begin{equation}\label{fixed_field_eq}
        \rank(M^H) \leq \rank((A(L_i)/A(K))^H).
    \end{equation}

    Consider the exact sequences 
    \[0 \longrightarrow A(K) \longrightarrow \mathrm{Mor}_K(X,A) \longrightarrow M \longrightarrow 0\]
    and 
    \[0 \longrightarrow A(K) \longrightarrow A(L_i) \longrightarrow A(L_i)/A(K) \longrightarrow 0.\]
    By taking $H$-invariants, we obtain the exact sequences
    \[0 \longrightarrow A(K) \longrightarrow \mathrm{Mor}_K(X/H,A) \longrightarrow M^H \longrightarrow H^1(H,A(K))\]
    and 
    \[0 \longrightarrow A(K) \longrightarrow A(L_i^H) \longrightarrow (A(L_i)/A(K))^H \longrightarrow H^1(H,A(K)).\]
    Since $H$ is finite and $A(K)$ is finitely generated, $H^1(H,A(K))$ is finite. Therefore, 
    \[\rank(M^H) = \rank(M_H)\]
    and 
    \[\rank((A(L_i)/A(K))^H) = \rank(A(L_i^H)) - \rank(A(K)).\]
    From \eqref{fixed_field_eq}, we conclude that 
    \[\rank(A(L_i^H)) \geq \rank(A(K)) + \rank(M_H).\]
\end{proof}

\section{Explicit lower bounds for rank growth}\label{Explicit_lower_bound}

Theorem~\ref{main_theorem} reduces the study of rank growth of abelian varieties to the study of the $G$-module 
\[M := \mathrm{Mor}_K(X,A)/A(K).\]
Indeed, for infinitely many mutually linearly disjoint $G$-extensions $L_i/K$, we have an injection 
\[M \hookrightarrow A(L_i)/A(K).\] 
Thus, any lower bound for the rank of $M$ yields a corresponding lower bound for rank growth. In this section, we obtain explicit lower bounds by studying the $G$-submodule~$N$ of $M$ generated by a single morphism $f:X \rightarrow A$ over $K$.

\begin{corollary}\label{N_corollary}
    Let $A/K$ be an abelian variety over a number field $K$. Let $X/K$ be a smooth projective variety with a finite subgroup $G \subseteq \Aut_K(X)$ such that $X/G \cong \bbp^d_K$. Let $f:X \rightarrow A$ be a morphism over $K$. Define 
    \begin{equation}\label{M_def}
        M := \mathrm{Mor}_K(X,A)/A(K),
    \end{equation}
    where we identify $A(K)$ with constant morphisms $X \rightarrow A$ over $K$, and let $N$ be the $G$-submodule of $M$ generated by $f \bmod{A(K)}$.
    
    Then there exist infinitely many mutually linearly disjoint $G$-extensions $L_i/K$ such that 
    \[N \hookrightarrow A(L_i)/A(K)\]
    as $G$-modules and hence 
    \[\rank(A(L_i)) \geq \rank(A(K)) + \rank(N).\]
\end{corollary}
\begin{proof}
    Since $N$ is a $G$-submodule of $M$, the corollary follows from Theorem~\ref{main_theorem}.
\end{proof}

To study the $G$-module $N$, we first record that $M$ is torsion-free. 

\begin{proposition}\label{M_torsion_free}
    The abelian group $M$ given in \eqref{M_def} is torsion-free.
\end{proposition}
\begin{proof}
    Let $f \in \mathrm{Mor}_K(X,A)$ and suppose that $nf \in A(K)$ for some positive integer $n$. Then the morphism $[n]f:X \rightarrow A$ is constant with image $Q \in A(K)$. Since $[n]:A \rightarrow A$ is finite, $[n]^{-1}(Q)$ is a finite set. Since $X$ is connected, $f:X \rightarrow A$ must be constant. Hence, $f \in A(K)$.
\end{proof}

Consider the $G$-equivariant homomorphism 
\[\phi : \bbz[G] \longrightarrow M,\quad \sum_{g \in G} a_g g \longmapsto \sum_{g \in G} a_g(f \circ g^{-1}) \bmod{A(K)}.\]
Its image is precisely $N$. Tensoring $\phi$ with $\bbq$, we obtain the $G$-equivariant homomorphism 
\begin{equation}\label{Phi_definition}
    \Phi : \bbq[G] \longrightarrow M \otimes_\bbz \bbq.
\end{equation}
Since $M$ is torsion-free by Proposition~\ref{M_torsion_free}, we obtain the following commutative diagram:
\[
\begin{tikzcd} 
\bbz[G] \arrow[r, "\phi"] \arrow[d, hook] & M \arrow[d, hook] \\ \bbq[G] \arrow[r, "\Phi"] & M \otimes_\bbz \bbq.
\end{tikzcd}
\]
Hence, the image of $\Phi$ is precisely $N \otimes_\bbz \bbq$. Consequently, computing the rank of $N$ reduces to determining the kernel of $\Phi$ in $\bbq[G]$.

We begin with cyclic subgroups of $G$. Let $\sigma \in G$ be an automorphism of order $n$. We will compute the rank of the subgroup $P$ of $N$ spanned by 
\[f \bmod{A(K)}, \quad f \circ \sigma \bmod{A(K)}, \quad \ldots, \quad f \circ \sigma^{n-1} \bmod{A(K)}\]
in terms of the minimal polynomial of $f$ with respect to $\sigma$.

We first establish some notation. For an arbitrary morphism $f:X \rightarrow A$ and an integral polynomial 
\[h(T)=a_0 + a_1T + \cdots + a_mT^m \in \bbz[T],\]
we define
\[f \circ h(\sigma) := a_0f + a_1f \circ \sigma + \cdots + a_mf \circ \sigma^m.\]
This defines the $\bbz[T]$-action on $\mathrm{Mor}_K(X,A)$, hence on $M$. The $\bbz[T]$-action on $M$ naturally extends to the $\bbq[T]$-action on $M \otimes_\bbz \bbq$. It is easy to see that $M \otimes_\bbz \bbq$ becomes a $\bbq[T]$-module under this action. 

For a fixed morphism $f:X \rightarrow A$, let $I$ be the annihilator ideal of $f \bmod{A(K)}$ in the $\bbq[T]$-module $M \otimes_\bbz \bbq$. Since $\sigma$ has order $n$, $T^n-1 \in I$. Therefore, $I$ is nonzero and since $\bbq[T]$ is a principal ideal domain, $I$ has a unique monic generator $p(T)$. By Gauss's lemma with $p(T) \mid T^n-1$, we conclude that $p(T)$ is integral. We call the unique $p(T) \in \bbz[T]$ the minimal polynomial of $f$ with respect to $\sigma$.

Recall the $G$-equivariant homomorphism given in \eqref{Phi_definition}, 
\[\Phi : \bbq[G] \longrightarrow M \otimes_\bbz \bbq.\]
If we compose $\Phi$ with 
\[\bbq[T] \longrightarrow \bbq[G],\quad T \longmapsto \sigma^{-1},\]
then we obtain the following homomorphism 
\[\Psi : \bbq[T] \longrightarrow M \otimes_\bbz \bbq.\]
The annihilator ideal $I$ of $f \bmod{A(K)}$ in the $\bbq[T]$-module $M \otimes_\bbz \bbq$ is exactly the kernel of~$\Psi$.

\begin{corollary}\label{cyclic_N_corollary}
    Keep the hypotheses and notation of Corollary~\ref{N_corollary}, and let $\sigma \in G$ be an automorphism of order $n$. Let $p(T)$ be the minimal polynomial of $f$ with respect to~$\sigma$ and $r$ be the degree of $p(T)$.
     Then there exist infinitely many mutually linearly disjoint $G$-extensions $L_i$ over $K$ such that 
    \[\rank(A(L_i))\geq \rank(A(K))+r.\]
\end{corollary}
\begin{proof}
    Let $P$ be the subgroup of $N$ spanned by
    \[f \bmod{A(K)}, ~~~ f \circ \sigma \bmod{A(K)},~~~\ldots,~~~f \circ \sigma^{n-1} \bmod{A(K)}.\]
    Then $P \otimes_\bbz \bbq$ is the image of $\Psi$. By the first isomorphism theorem, $P \otimes_\bbz \bbq \cong \bbq[T]/I$. Since $p(T)$ has degree $r$, $P \otimes_\bbz \bbq \cong \bbq[T]/I$ has $\bbq$-dimension $r$, and hence $P$ has rank $r$. The conclusion now follows from Corollary~\ref{N_corollary}.
\end{proof}

When $\sigma$ has prime order, the factorization of $T^p-1$ leaves only four possible minimal polynomials. This yields the following explicit criteria.

\begin{corollary}\label{p_cyclic_N_corollary}
    Keep the hypotheses and notation of Corollary~\ref{cyclic_N_corollary}, and assume further that $\sigma \in G$ has a prime order $p$. 
    \begin{enumerate}[\normalfont (a)]
        \item If $f \circ \sigma - f$ is non-constant, then there exist infinitely many mutually linearly disjoint $G$-extensions $L_i$ over $K$ such that 
        \[\rank(A(L_i))\geq \rank(A(K))+(p-1).\]
        \item If $f \circ \sigma - f$ and $f \circ \sigma^{p-1} + \cdots + f \circ \sigma + f$ are both non-constant, then there exist infinitely many mutually linearly disjoint $G$-extensions $L_i$ over $K$ such that 
        \[\rank(A(L_i))\geq \rank(A(K))+p.\]
    \end{enumerate}
\end{corollary}
\begin{proof}
    Let $p(T)$ be the minimal polynomial of $f \bmod{A(K)}$ in the $\bbq[T]$-module $M \otimes_\bbz \bbq$, and let $r$ be its degree. Note that $p(T) \mid T^p-1$. From the decomposition 
    \[T^p-1 = (T-1)(T^{p-1}+\cdots+T+1),\]
    there are four possibilities for $p(T)$:
    \[1,\quad T-1,\quad T^{p-1}+\cdots+T+1,\quad T^p-1.\]
    \begin{enumerate}[(a)]
        \item If $f \circ \sigma - f$ is non-constant, then $p(T) \nmid T-1$, and hence $r \geq p-1$.
        \item If $f \circ \sigma - f$ and $f \circ \sigma^{p-1} + \cdots + f \circ \sigma + f$ are both non-constant, then $p(T) \nmid T-1$ and $p(T) \nmid T^{p-1}+\cdots+T+1$, and hence $r=p$.
    \end{enumerate}
    The conclusion now follows from Corollary~\ref{cyclic_N_corollary}.
\end{proof}

We next consider the case where $G \cong \cS_n$ and $f$ is invariant under a subgroup $H \cong \cS_{n-1}$ of $G$.

\begin{corollary}\label{S_n_N_corollary}
    Keep the hypotheses and notation of Corollary~\ref{N_corollary}, suppose $G \cong \cS_n$, that $f$ is non-constant, and that $f$ is invariant under $H \cong \cS_{n-1}$.
Then there exist infinitely many mutually linearly disjoint $\cS_n$-extensions $L_i/K$ such that 
    \[\rank(A(L_i))\geq \rank(A(K))+(n-1).\]
\end{corollary}
\begin{proof}
    Let $V=\bbq[G/H]$ be the rational permutation representation of $G$. Since $f$ is invariant under $H$, the $G$-equivariant homomorphism \eqref{Phi_definition} 
    \[\Phi : \bbq[G] \longrightarrow M \otimes_\bbz \bbq\]
    factors through $V$. Let 
    \[\Psi : V \longrightarrow M \otimes_\bbz \bbq\]
    be the induced $G$-equivariant homomorphism. If we let $\{g_1,\ldots,g_n\}$ to be the representatives of $G/H$, then $\Psi$ is explicitly defined by 
    \[\Psi : g_i \longmapsto f \circ g_i^{-1} \bmod{A(K)}.\]

    We first prove that $\Psi(t) = 0$ where 
    \[t := \sum_{i=1}^n g_i \in V.\]
    Consider 
    \begin{equation}\label{constant_morphism_eq}
        f \circ g_1^{-1} + \cdots + f \circ g_n^{-1} : X \longrightarrow A.
    \end{equation}
   Note that \eqref{constant_morphism_eq} is invariant under $G$, so \eqref{constant_morphism_eq} factors through $X/G \cong \bbp_K^d$. However, every morphism from $\mathbb P^d$ to an abelian variety is constant (\cite[Corollary~3.9]{Mil86}), so \eqref{constant_morphism_eq} is constant. It follows that $\Psi(t)=0$.

    Recall that $V$ decomposes into two irreducible representations. Explicitly, 
    \[V \cong T \oplus S,\]
    where 
    \[T := \bbq t = \{at\:|\:a \in \bbq\}\]
    is the trivial representation and 
    \[S := \left\{\sum_{i=1}^n a_ig_i \:|\:\sum_{i=1}^n a_i = 0\right\}\]
    is the standard representation. Since $\Psi(t)=0$, $\Psi$ is zero on $T$. Since $f$ is non-constant and $M$ is torsion-free by Proposition~\ref{M_torsion_free}, $\Psi$ is nonzero. It follows that $\Psi$ is nonzero on~$S$. Since $S$ is irreducible, $\Psi(S)$ has dimension $n-1$. Hence, 
    \[\Psi(V) \cong \Psi(T) \oplus \Psi(S)\]
    has dimension $n-1$.

    Since the image of $\Psi$ is the image of $\Phi$, which is $N \otimes_\bbz \bbq$, the conclusion follows from Corollary~\ref{N_corollary}.
\end{proof}

\section{Rank growth of Jacobian varieties}\label{section_Jacobian}

In this section, we apply Theorem~\ref{main_theorem} to Jacobian varieties. Let $X/K$ be a smooth projective curve with $X(K)$ nonempty and $J/K$ be the Jacobian variety of $X$. The universal property of the Jacobian allows us to identify the $G$-module 
\[M := \mathrm{Mor}_K(X,A)/A(K)\]
with the $G$-module of group homomorphisms from $J$ to $A$. 

\begin{corollary}\label{Jacobian_abelian_variety_corollary}
    Let $A/K$ be an abelian variety over a number field $K$. Let $X/K$ be a smooth projective curve with a finite subgroup $G \subseteq \Aut_K(X)$ such that $X/G \cong \bbp_K^1$. Let $X(K)$ be nonempty and $J/K$ be the Jacobian variety of $X$. 
    
    Then there exist infinitely many mutually linearly disjoint $G$-extensions $L_i/K$ such that 
    \[\mathrm{Hom}_K(J,A) \hookrightarrow A(L_i)/A(K)\]
    as $G$-modules and hence 
    \[\rank(A(L_i)) \geq \rank(A(K)) + \rank(\mathrm{Hom}_K(J,A)).\]
\end{corollary}
\begin{proof}
    We show that 
    \[M := \mathrm{Mor}_K(X,A)/A(K) \cong \mathrm{Hom}_K(J,A)\]
    as $G$-modules. Then the conclusion follows from Theorem~\ref{main_theorem}.

    By the universal property of the Jacobian, there exists a surjective homomorphism 
    \[\Omega : \mathrm{Mor}_K(X,A) \longrightarrow \mathrm{Hom}_K(J,A)\]
    with $\ker \Omega = A(K)$. We show that $\Omega$ is $G$-equivariant. Recall that $G$ acts on $\mathrm{Mor}_K(X,A)$ by 
    \[g \cdot f = f \circ g^{-1}.\]
    Again by the universal property of the Jacobian, there exists a group homomorphism $\iota:\Aut_K(X) \rightarrow \Aut_K(J)$. Then $G$ acts on $\mathrm{Hom}_K(J,A)$ by 
    \[g \cdot F = F \circ \iota(g)^{-1}.\]
    Given $g \in G$ and $f \in \mathrm{Mor}_K(X,A)$, we have 
    \[\Omega(g \cdot f) = \Omega(f \circ g^{-1}) = \Omega(f) \circ \iota(g)^{-1} = g \cdot \Omega(f).\]
    Therefore, $\Omega$ is a surjective $G$-equivariant homomorphism with $\ker \Omega = A(K)$.
\end{proof}

\begin{corollary}\label{Jacobian_corollary}
    Let $X/K$ be a smooth projective curve over a number field $K$ with a finite subgroup $G \subseteq \Aut_K(X)$ such that $X/G \cong \bbp_K^1$ and $X(K)$ nonempty. Let $J/K$ be the Jacobian variety of $X$.
Then there exist infinitely many mutually linearly disjoint $G$-extensions $L_i/K$ such that 
    \[\End_K(J) \hookrightarrow J(L_i)/J(K)\]
    as $G$-modules and hence 
    \[\rank(J(L_i)) \geq \rank(J(K)) + \rank(\End_K(J)).\]
\end{corollary}
\begin{proof}
    It follows from Corollary~\ref{Jacobian_abelian_variety_corollary} by taking $A=J$.
\end{proof}

\begin{corollary}\label{Jacobian_factor_corollary}
    Keep the hypotheses and notation of Corollary~\ref{Jacobian_corollary} and suppose 
    \[J \cong_K A^n \times B\]
    for abelian varieties $A,B$ and positive integer $n$.
    Then there exist infinitely many mutually linearly disjoint $G$-extensions $L_i/K$ such that 
    \[\rank(A(L_i)) \geq \rank(A(K)) + n \cdot \rank(\End_K(A)).\]
\end{corollary}
\begin{proof}
    In the isogeny category, 
    \[\mathrm{Hom}_K(J,A) \otimes_\bbz \bbq\]
    contains 
    \[\End_K(A)^n \otimes_\bbz \bbq.\]
    Therefore, 
    \[\rank(\mathrm{Hom}_K(J,A)) \geq n \cdot \rank(\End_K(A)).\]
    Now the proof is completed by Corollary~\ref{Jacobian_abelian_variety_corollary}.
\end{proof}

As the first application, we consider hyperelliptic curves. 

\begin{corollary}\label{hyperelliptic_corollary}
    Let $X/K$ be a hyperelliptic curve with $X(K)$ nonempty. Let $J/K$ be the Jacobian variety of $X$. Then there exist infinitely many mutually linearly disjoint quadratic extensions $L_i/K$ such that 
    \[\rank(J(L_i)) \geq \rank(J(K)) + \rank(\End_K(J)).\]
\end{corollary}
\begin{proof}
    Let $\iota \in \Aut_K(X)$ be the hyperelliptic involution. Then $X/\langle \iota \rangle$ has genus zero. Since $X(K)$ is nonempty, so is $(X/\langle \iota \rangle)(K)$, and hence $X/\langle \iota \rangle \cong \bbp_K^1$. Now the conclusion follows from Corollary~\ref{Jacobian_corollary}.
\end{proof}

We next consider modular curves $X(N)$. Their definition and basic properties are discussed in \cite{DS05}.

\begin{corollary}\label{modular_curve}
    Let $N \geq 3$ be an integer and let $K$ be a number field containing a primitive $N$th root $\zeta_N$. Let $X(N)/K$ be the modular curve and let $J(N)/K$ be the Jacobian variety of $X(N)$.    Then there exist infinitely many mutually linearly disjoint $\mathrm{PSL}_2(\bbz/N\bbz)$-extensions $L_i/K$ such that
    \[\rank (J(N)(L_i)) \geq \rank (J(N)(K))+\rank(\End_K(J(N))).\]
\end{corollary}
\begin{proof}
    The group $\mathrm{PSL}_2(\bbz/N\bbz)$ acts on $X(N)$ over $K$ and the quotient $X(N)/\mathrm{PSL}_2(\bbz/N\bbz)$ is isomorphic to $X(1) \cong \bbp_K^1$. For details, see \cite{DS05}.
\end{proof}

We finally consider Jacobians with complex multiplication. If $X/K$ has genus $g$ and $J/K$ has complex multiplication over $K$, then 
\[\rank(\End_K(J)) \geq 2g.\]
As an example, we consider Fermat curves.

For an integer $n \geq 3$, let $X_n/\bbq$ be the smooth projective curve in $\bbp_\bbq^2$ defined by 
\begin{equation}\label{Fermat_curve}
    x^n+y^n = z^n.
\end{equation}
Note that $X_n(\bbq)$ is nonempty. Let $J_n/\bbq$ be the Jacobian variety of $X_n$. 

Let $K$ be a number field containing $\zeta_n$. The curve $X_n/K$ admits natural cyclic automorphisms
\[\sigma_n : (x:y:z) \longmapsto (\zeta_n x : y : z),\quad \tau_n : (x:y:z) \longmapsto (x:\zeta_n y:z).\]

\begin{corollary}\label{Fermat_corollary}
    Let $n \geq 3$ be an integer and let $K$ be a number field containing $\zeta_n$. Let $X_n/K$ be the Fermat curve defined by \eqref{Fermat_curve} and let $J_n/K$ be the Jacobian variety of $X_n$. 
    Then there exist infinitely many mutually linearly disjoint $C_n$-extensions $L_i/K$ such that
    \[\rank (J_n(L_i)) \geq \rank (J_n(K))+(n-1)(n-2).\]
\end{corollary}
\begin{proof}
    Let $g$ be the genus of $X_n/K$. By the genus-degree formula, $g = (n-1)(n-2)/2$. By \cite[p.91]{Lim91}, $J_n/K$ has complex multiplication over $K$ so that 
    \[\rank(\End_K(J_n)) \geq 2g = (n-1)(n-2).\]
    Finally, note that $X_n/\langle \sigma_n \rangle \cong \bbp_K^1$ and $\langle \sigma_n \rangle \cong C_n$. Now the conclusion follows from Corollary~\ref{Jacobian_corollary}.
\end{proof}

\section{Explicit examples}\label{section_A4}

In this section, we illustrate Theorem~\ref{main_theorem} through two explicit examples. Our primary example focuses on the Legendre family, demonstrating a rank growth of at least~$3$ over infinitely many mutually linearly disjoint $\cA_4$-extensions. We conclude with an example arsing from the modular curve $X(7)$.

Rather than starting from a given abelian variety, we proceed in the opposite direction: we first choose a (hyperelliptic) curve $X$ with a prescribed automorphism group $G$ (which can be constructed using~\cite{Sha03}), and then obtain an elliptic curve $E$ as a suitable quotient of $X$ by $H \leq \Aut_K(X)$. This approach is particularly advantageous for applying Corollary~\ref{cyclic_N_corollary}, as illustrated below.

\begin{proposition}\label{non_constant_proposition}
    Let $X/K$ be a smooth projective curve over a number field $K$. Let $H \subseteq \Aut_K(X)$ be a finite subgroup such that $X/H$ is an elliptic curve $E$. Let $\sigma \in \Aut_K(X)$ be an automorphism of order $n$ and $f:X \rightarrow X/H = E$ be the quotient morphism.
    \begin{enumerate}[\normalfont (a)]
        \item Let $G_1$ be the subgroup of $\Aut_K(X)$ generated by $H$ and $\bigcap\limits_{i=1}^{n-1}\sigma^{-i}H\sigma^i$. If $X/G_1$ has genus~$0$, then $f \sigma^{n-1} + \cdots + f \sigma + f$ is non-constant.
        \item Let $G_2$ be the subgroup of $\Aut_K(X)$ generated by $\sigma$ and $H$. If $X/G_2$ has genus~$0$, then $f \sigma - f$ is non-constant.
    \end{enumerate}
\end{proposition}
\begin{proof}
    (a) Suppose for the sake of contradiction that $f \sigma^{n-1} + \cdots + f \sigma + f$  is constant. Then for any
    \[g \in \bigcap_{i=1}^{n-1}\sigma^{-i}H\sigma^i,\]
    we have
    \begin{align*}
        f\sigma^{n-1}+\cdots+f\sigma+f &= (f\sigma^{n-1}+\cdots+f\sigma+f)g \\
        &= f(\sigma^{n-1} g \sigma^{-(n-1)})\sigma^{n-1} + \cdots + f(\sigma g \sigma^{-1})\sigma + fg \\
        &= f\sigma^{n-1} + \cdots + f\sigma + fg.
    \end{align*}
    Thus $fg=f$, so $f$ is invariant under $\bigcap_{i=1}^{n-1}\sigma^{-i}H\sigma^i$. As $f$ is also $H$-invariant, it is invariant under $G_1$. Hence $f$ factors through $X/G_1$. Since $X/G_1$ has genus $0$ and $E$ has genus $1$, every morphism $X/G_1 \rightarrow E$ is constant. Therefore, $f$ is constant, which is a contradiction.

    (b) Suppose $f \sigma - f$ is constant. Then there exists $T \in E(\overline{K})$ such that
    \[f\sigma = f + T.\]
    Then, this implies
    \[f = f\sigma^n = f + [n]T,\]
    so $[n]T=O$. It follows that
    \[[n]f\sigma = [n]f,\]
    and hence $[n]f$ is $\sigma$-invariant. Since $f$ is $H$-invariant, so is $[n]f$, and therefore $[n]f$ is invariant under $G_2$. Thus $[n]f$ factors through $X/G_2$. As in part~(a), every morphism from a genus $0$ curve to $E$ is constant, so $[n]f$ is constant. Since $[n]:E \rightarrow E$ is finite and $X$ is connected, this forces $f$ itself to be constant, yielding a contradiction.
\end{proof}

\begin{corollary}\label{explicit_example_corollary}
    Let $X/K$ be a smooth projective curve over a number field $K$. Let $G,H \subseteq \Aut_K(X)$ be finite subgroups such that $X/G \cong \bbp_K^1$ and $X/H$ is an elliptic curve $E$. Let $\sigma \in G$ be an automorphism of prime order $p$ and $f:X \rightarrow X/H = E$ be the quotient morphism. Let $G_1,G_2 \subseteq \Aut_K(X)$ be as in Proposition~\ref{non_constant_proposition}.
    \begin{enumerate}[\normalfont (a)]
        \item If $X/G_2$ has genus $0$, then there exist infinitely many mutually linearly disjoint $G$-extensions $L_i$ over $K$ such that 
        \[\rank(E(L_i)) \geq \rank(E(K)) + (p-1).\]
        \item If $X/G_1$ and $X/G_2$ both have genus $0$, then there exist infinitely many mutually linearly disjoint $G$-extensions $L_i$ over $K$ such that 
        \[\rank(E(L_i)) \geq \rank(E(K)) + p.\]
    \end{enumerate}
\end{corollary}
\begin{proof}
    Apply Corollary~\ref{p_cyclic_N_corollary} with Proposition~\ref{non_constant_proposition}.
\end{proof}

We now apply Corollary~\ref{explicit_example_corollary} to the Legendre family. We construct a concrete hyperelliptic curve $X/K$ that satisfies the conditions of Corollary~\ref{explicit_example_corollary}. 

\begin{theorem}\label{rank_A4}
    Let $K$ be a number field containing $\sqrt{-1}$. For each $a \in K$ such that $a \neq 0,1$ and $a^2-a+1 \neq 0$, let
    \begin{equation}\label{Ea}
        E_a/K : y^2=x(x-1)(x-a)
    \end{equation}
    be an elliptic curve defined over $K$. If $K$ contains $\sqrt{a(a-1)(a^2-a+1)}$, then there exist infinitely many mutually linearly disjoint $\cA_4$-extensions $L_i$ over $K$ such that
    \[\rank(E_a(L_i))\geq \rank(E_a(K))+3.\]
\end{theorem}
\begin{proof}
    Fix $\lambda \in K$ such that $\lambda^2 + 108 \neq 0$. Let
    $X_\lambda/K$ be the curve whose affine equation is defined by
    \[X_\lambda : y^2 = x^{12}-\lambda x^{10}-33x^8+2\lambda x^6-33x^4-\lambda x^2+1.\]
    The discriminant of the polynomial on the right side is 
    \[2^{52}(\lambda^2+108)^8 \neq 0,\] 
    under the assumption, so $X_\lambda$ is a hyperelliptic curve of genus $5$. 
    
    The automorphism group $\Aut_K(X_\lambda)$ contains $C_2 \times \cA_4$ by \cite[Section~4.3]{Sha03}. Explicitly, the subgroup of $\Aut_K(X_\lambda)$ generated by the hyperelliptic involution $\iota$ defined by 
    \[\iota(x,y) = (x,-y)\]
    is isomorphic to $C_2$, and the subgroup generated by $\sigma$ and $\mu$, where
    \begin{equation}\label{sigma-mu2}
        \sigma(x,y) = \left(\sqrt{-1}\dfrac{x+1}{x-1}, \dfrac{8\sqrt{-1}y}{(x-1)^6}\right), \quad
        \mu(x,y) = \left(-x,-y\right),
    \end{equation}
    is isomorphic to $\cA_4$. We let $G=\langle \sigma,\mu \rangle \cong \cA_4$ and $H = \langle \iota\sigma^2\mu\sigma, \mu\rangle$. If we explicitly write $\iota\sigma^2\mu\sigma$, then
    \[ \iota\sigma^2\mu\sigma(x,y) = \left(\dfrac{1}{x},\dfrac{y}{x^6}\right).\]

    Let $E_\lambda'/K$ be an elliptic curve whose affine equation is defined by
    \[E_\lambda' : y^2=x^3-\lambda x^2-36x+4\lambda,\]
    noting that the discriminant is 
    \[2^{4}(\lambda^2+108)^2 \neq 0\] 
    under our assumption.

    Define a morphism $f_\lambda : X_\lambda \rightarrow E'_\lambda$ over $K$ by
    \[f_\lambda(x,y)=\left(x^2+\frac{1}{x^2}, \frac{y}{x^3}\right).\]
    This is the quotient map $X_\lambda \rightarrow X_\lambda/H$. 

    Note that a given elliptic curve $E_a/K :y^2=x(x-1)(x-a)$ has $j$-invariant, 
    \[2^8\dfrac{(a^2-a+1)^3}{a^2(a-1)^2}.\] 
    If we fix
    \[\lambda=\dfrac{2(a+1)(a-2)(2a-1)}{a(a-1)} = \dfrac{2(2a^3-3a^2-3a+2)}{a(a-1)},\]
    then 
    \[j(E_a)= 2^8\frac{(a^2-a+1)^3}{a^2(a-1)^2}=2^4(\lambda^2+108)=j(E'_\lambda).\]
    In fact, if we let
    \[u=\dfrac{4(a^2-a+1)}{a(a-1)},\]
    then $E_a \cong E'_{\lambda}$ over $K(\sqrt{u})$ via an isomorphism $\phi :E_a\rightarrow E'_{\lambda}$ defined by
    \[\phi(x,y)=\left(u x -\dfrac{2(a+1)}{a-1}, \sqrt{u}^3y \right).\]

    Now suppose $K$ contains $\sqrt{a(a-1)(a^2-a+1)}$. Then $\phi:E_a \rightarrow E'_{\lambda}$ is an isomorphism over $K$. Let $f = \phi^{-1} \circ f_\lambda : X_\lambda \rightarrow E_a$ and denote $X_\lambda$ by $X$. To apply Corollary~\ref{explicit_example_corollary}, we have to prove that $X/\cA_4 \cong \bbp^1_K$. Since $X$ has a $K$-rational point $(0,1)$, it suffices to prove that $X/\cA_4$ has genus~$0$.

    Let $g$ be the genus of $X/\cA_4$. Since $X$ is a hyperelliptic curve, $X/(C_2 \times \cA_4)$ has genus~$0$. By applying the Riemann-Hurwitz formula to
    \[\iota : X/\cA_4 \rightarrow X/(C_2 \times \cA_4),\]
    we obtain 
    \[2g-2 = 2(-2) + \sum_{P \in X/\cA_4} (e_P-1)\]
    where $e_P$ is a ramification index at $P$. As $\iota$ is a map of degree~$2$, the summation over $P \in X/\cA_4$ on the right equals the number $N$ of points ramified under $\iota$. The number of such points is equal to the number of $(C_2 \times \cA_4)$-orbits of $\overline{K}$-rational points of $X$ which are the same as $\cA_4$-orbits. Some calculations show that there are exactly two such orbits which are
    \[\{(\alpha_1,0),(\alpha_2,0),\ldots,(\alpha_{12},0)\}\]
    where $\alpha_i$ are distinct roots of the polynomial 
    \[x^{12}-\lambda x^{10}-33 x^8+2\lambda x^6-33 x^4-\lambda x^2+1\] 
    and
    \[\{(0,\pm 1),\pm \infty,(\pm \sqrt{-1},\pm 8\sqrt{-1}),(\pm 1, \pm 8\sqrt{-1})\}\]
    where $\pm \infty$ are the two infinity points of $X$. Therefore $N=2$ and $g=0$. 

    Recall that $f$ is the quotient map $X \rightarrow X/H$ where $H$ is generated by $\iota\sigma^2\mu\sigma$ and $\mu$. Let $G_1,G_2$ be defined by 
    \[G_1 := \langle \sigma^{-1}H\sigma \cap \sigma^{-2}H\sigma^2, H \rangle,\quad G_2 := \langle \sigma,H \rangle.\]
    It is obvious that $G_2 = C_2 \times \cA_4$, so $X/G_2$ has genus 0. Next, note that $\iota\mu \in \sigma^{-1}H\sigma \cap \sigma^{-2}H\sigma^2$. Thus
    \[\iota \in \langle \sigma^{-1}H\sigma \cap \sigma^{-2}H\sigma^2, H \rangle = G_1,\]
    so that $X/G_1$ has genus 0. Now by applying Corollary~\ref{p_cyclic_N_corollary} finally, the proof completes.
\end{proof}

We conclude this section with a complementary example in which the rank bound is obtained from the decomposition of the modular Jacobian. 

\begin{theorem}
    Let $K$ be a number field containing $\zeta_7$. Let 
    \[E/K : y^2+3xy+y=x^3-2x-3\]
    be an elliptic curve over $K$. Then there exist infinitely many mutually linearly disjoint $\mathrm{PSL}_2(\bbz/7\bbz)$-extensions $L_i/K$ such that 
    \[\rank(E(L_i)) \geq \rank(E(K)) + 6.\]
\end{theorem}
\begin{proof}
    Consider the modular curve~$X(7)$ and its Jacobian $J(7)$. By \cite[p.12]{BKX13}, the Jacobian $J(7)$ of the modular curve $X(7)$ is isomorphic to $E^3$ over $K$. By Corollary~\ref{modular_curve} and Corollary~\ref{Jacobian_factor_corollary}, and noting that $E/K$ has complex multiplication over $K$ ($j(E)=-3375$), the conclusion follows.
\end{proof}

\section{Rank growth over symmetric extensions}\label{section_Sn}

In this section, we apply Corollary~\ref{S_n_N_corollary} to obtain rank growth of abelian varieties over $\cS_n$-extensions. We show that every abelian variety admits a threshold beyond which the rank growth of at least $n$ occurs over infinitely many mutually linearly disjoint $\cS_{n+1}$-extensions. We then discuss the least such threshold.

\begin{lemma}\label{curve_monodromy_group}
    Let $X/K$ be a smooth projective curve of genus $g \geq 1$ over a number field $K$ with $X(K)$ nonempty. Then for every $d \geq 2g+1$, there exists a finite morphism $p:X \rightarrow \bbp_K^1$ over $K$ of degree $d$ with monodromy group $\cS_d$.
\end{lemma}
\begin{proof}
    Take $P \in X(K)$ and let $\mathcal{L}$ be the line bundle corresponding to the divisor $dP$. Since $d \geq 2g+1$, $\mathcal{L}$ is very ample. Therefore, $\mathcal{L}$ defines a projective embedding $X \hookrightarrow \bbp_K^r$ and the degree of $X$ in $\bbp_K^r$ is $d$. 
    
    Consider the Grassmannian $\mathbb{G}(r-2,\bbp_K^r)$ of $(r-2)$-planes of $\bbp_K^r$. By \cite[Theorem~1.1]{PS05}, after base change to $\overline{K}$, the closure of the locus of non-uniform $(r-2)$-planes has codimension at least 2. Since non-uniformity is invariant under $\Gal(\overline{K}/K)$, this closure is defined over $K$. Removing also the proper closed locus of $(r-2)$-planes meeting $X$, we obtain a nonempty open subset of $\mathbb{G}(r-2,\bbp_K^r)$ defined over $K$. Choose a $K$-rational point $L$ of this subset.

    Now let $p:X \rightarrow \bbp_K^1$ be the morphism induced by the projection from $L$. Then by our choice of $L$, $p$ has degree $d$ and monodromy group $\cS_d$.
\end{proof}

\begin{lemma}\label{S_d_extension}
    Let $A/K$ be an abelian variety over a number field $K$. Let $X/K$ be a smooth projective curve, $j:X \rightarrow A$ be a non-constant morphism over $K$, and $p:X \rightarrow \bbp_K^1$ be a finite morphism over $K$ of degree $d$ with monodromy group $\cS_d$.

    Then there exist infinitely many mutually linearly disjoint $\cS_d$-extensions $L_i/K$ such that 
    \[\rank(A(L_i))\geq \rank(A(K))+(d-1).\]
\end{lemma}
\begin{proof}
    Let $q:Y \rightarrow \bbp_K^1$ be the Galois closure of $p:X \rightarrow \bbp_K^1$. Then $Y/G \cong \bbp_K^1$ and $Y/H \cong X$ for some $H \leq G \leq \Aut_K(Y)$, where 
    \[G \cong \cS_d,\quad H \cong \cS_{d-1}.\]
    Let $\pi:Y \rightarrow X$ be the quotient morphism and define 
    \[f:Y \xrightarrow{\pi} X \xrightarrow{j} A.\]
    Since $j$ is non-constant, so is $f$. Now apply Corollary~\ref{S_n_N_corollary} to $Y$ and $f$.
\end{proof}

\begin{theorem}\label{abelian_variety_S_n_extension}
    Let $A/K$ be an abelian variety over a number field $K$. Let $X/K$ be a smooth projective curve of genus $g \geq 1$ with $X(K)$ nonempty and $j:X \rightarrow A$ be a non-constant morphism over $K$. Then for every $n \geq 2g$, there exist infinitely many mutually linearly disjoint $\cS_{n+1}$-extensions $L_i/K$ such that 
    \[\rank(A(L_i)) \geq \rank(A(K)) + n.\]
\end{theorem}
\begin{proof}
    The theorem follows from Lemmas~\ref{curve_monodromy_group} and \ref{S_d_extension}.
\end{proof}

\begin{theorem}\label{N_A_existence}
    Let $A/K$ be an abelian variety over a number field $K$. Then there exists a positive integer $N$ such that for every $n \geq N$, there exist infinitely many mutually linearly disjoint $\cS_{n+1}$-extensions $L_i/K$ such that 
    \[\rank(A(L_i)) \geq \rank(A(K)) + n.\]
\end{theorem}
\begin{proof}
    Let $d$ be the dimension of $A$. Fix $P \in A(K)$ and a projective embedding $A \hookrightarrow \bbp_K^r$. By Bertini's theorem with $P \in A(K)$, we can choose general hypersurfaces $H_1,\ldots,H_{d-1}$ defined over $K$, all containing $P$, such that 
    \[X := A \cap H_1 \cap \cdots \cap H_{d-1}\]
    is a smooth projective curve. Then $P \in X(K)$, so $X(K)$ is nonempty. Let $g$ be the genus of $X/K$ and $j:X \rightarrow A$ be the inclusion. Then Theorem~\ref{abelian_variety_S_n_extension} gives the conclusion with $N = 2g$. 
\end{proof}

Combining Theorem~\ref{N_A_existence} with Corollary~\ref{main_theorem_fixed_field_corollary}, we also obtain the following theorem.

\begin{theorem}
    Let $A/K$ be an abelian variety over a number field $K$. Then there exists a positive integer $N$ such that for every $n \geq N$, there exist infinitely many degree-$n$ extensions $F_i/K$ with Galois closure $\cS_n$ such that 
    \[\rank(A(F_i)) > \rank(A(K)).\]
\end{theorem}
\begin{proof}
    Recall that we have a smooth projective curve $X/K$ of genus $g$ in $A/K$. Take $N=2g+1$ and suppose $n \geq N$. By Lemma~\ref{curve_monodromy_group}, there exists a finite morphism $p:X \rightarrow \bbp_K^1$ of degree $n$ with monodromy group $\cS_n$. As in the proof of Lemma~\ref{S_d_extension}, let $q:Y \rightarrow \bbp_K^1$ be the Galois closure of $p:X \rightarrow \bbp_K^1$, so that $Y/G \cong \bbp_K^1$ and $Y/H \cong X$ for some $H \leq G \leq \Aut_K(Y)$, where 
    \[G \cong \cS_n,\quad H \cong \cS_{n-1}.\]
    By Corollary~\ref{main_theorem_fixed_field_corollary}, there exist infinitely many mutually linearly disjoint $G$-extensions $L_i/K$ such that 
    \[\rank(A(L_i^H)) > \rank(A(K))\]
    If we take $F_i = L_i^H$, then the proof is completed.
\end{proof}

In view of Theorem~\ref{N_A_existence}, we can define the following invariant.

\begin{definition}
    Let $A/K$ be an abelian variety over a number field $K$. We define $N_{A/K}$ to be the least positive integer $N$ such that for every $n \geq N$, there exist infinitely many mutually linearly disjoint $\cS_{n+1}$-extensions $L_i/K$ such that 
    \[\rank(A(L_i)) \geq \rank(A(K)) + n.\]
\end{definition}

Theorem~\ref{abelian_variety_S_n_extension} shows that 
\[N_{A/K} \leq 2g,\]
whenever there exists a smooth projective curve $X/K$ of genus $g \geq 1$ with $X(K)$ nonempty and a non-constant morphism $j:X \rightarrow A$ over $K$. In particular, this yields the following theorem.

\begin{theorem}\label{elliptic_to_abelian_N_1}
    Let $A/K$ be an abelian variety over a number field $K$. Suppose that there exists an elliptic curve $E/K$ and a non-constant morphism $j:E \rightarrow A$ over $K$. Then $N_{A/K}=1$.
\end{theorem}
\begin{proof}
    By Theorem~\ref{abelian_variety_S_n_extension}, we have $N_{A/K} \leq 2$. To prove $N_{A/K}=1$, we have to prove that there exist infinitely many mutually linearly disjoint $\cS_2$-extensions $L_i/K$ such that 
    \[\rank(A(L_i)) \geq \rank(A(K))+1.\]
    This follows from Lemma~\ref{S_d_extension} with the morphism $p:E \rightarrow E/\langle \pm 1 \rangle \cong \bbp_K^1$.
\end{proof}

\begin{corollary}\label{elliptic_N_1}
    Let $E/K$ be an elliptic curve over a number field $K$. Then $N_{E/K}=1$.
\end{corollary}

It is natural to ask whether $N_{A/K}=1$ holds for every abelian variety $A/K$ over every number field $K$. Observe that this would imply that there exist infinitely many mutually linearly disjoint quadratic extensions $L_i/K$ such that 
\[\rank(A(L_i)) \geq \rank(A(K))+1.\]
Therefore, the rank of $A(K^{(2)})$ will be infinite where $K^{(2)}$ denotes the compositum of all quadratic extensions of $K$. In particular, this would imply the Frey--Jarden conjecture, which we introduce later. Thus, the assertion $N_{A/K}=1$ appears to be significantly stronger than what is currently known.

Instead, it seems more reasonable to seek effective and sharp upper bounds for $N_{A/K}$.

\begin{question}
    Can we determine effective and sharp upper bounds for $N_{A/K}$?
\end{question}

\section{Connection with infinite rank conjectures}\label{section_infinite_rank}

In this section, we apply Theorem~\ref{main_theorem} to infinite rank conjectures. For convenience, we fix an algebraic closure $\overline{K}$ of $K$ and let $G_K := \Gal(\overline{K}/K)$ be the absolute Galois group of $K$. 

\subsection{Frey--Jarden conjecture}

Let $A/K$ be an abelian variety over a number field $K$. In \cite[p.127, Problem]{FJ74}, Frey and Jarden asked whether the rank of $A(K^{\mathrm{ab}})$ is infinite, where $K^{\mathrm{ab}}$ is the maximal abelian extension of $K$.

\begin{conjecture}[Frey--Jarden]\label{Frey_Jarden}
    Let $A/K$ be an abelian variety over a number field $K$. Then the rank of $A(K^{\mathrm{ab}})$ is infinite.
\end{conjecture}

Several important cases of Conjecture~\ref{Frey_Jarden} are known. Rosen and Wong \cite{RW02} proved Conjecture~\ref{Frey_Jarden} for Jacobian varieties of cyclic covers of $\bbp^1$. Petersen \cite{Pet06} generalized this result to quotients of Jacobian varieties of abelian covers of $\bbp^1$. Sairaiji and Yamauchi~\cite{SY12} proved Conjecture~\ref{Frey_Jarden} for Jacobian varieties $J/K$ of curves $C/K$ when $C(K^{\mathrm{ab}})$ is infinite. They also proved Conjecture~\ref{Frey_Jarden} for abelian varieties over $\bbq$ of $\mathrm{GL}_2$-type with trivial character. One can find more information about Conjecture~\ref{Frey_Jarden} in~\cite{SY12}.

We now present our application of Theorem~\ref{main_theorem} to Conjecture~\ref{Frey_Jarden}.

\begin{lemma}\label{compositum_lemma}
    Let $A/K$ be an abelian variety over a number field $K$. Let $L_i/K$ be infinitely many mutually linearly disjoint finite extensions such that 
    \[\rank(A(L_i)) > \rank(A(K)).\]
    Let $L/K$ be the compositum of all extensions $L_i/K$ in $\overline{K}$. Then the rank of $A(L)$ is infinite.
\end{lemma}
\begin{proof}
    For each $i$, take $P_i \in A(L_i)$ which is linearly independent from $A(K)$. Assume the rank of $A(L)$ is finite. Then there exists a positive integer $N$ and integers $a_1,\ldots,a_N$, not all zero, such that 
    \begin{equation}\label{lineary_dependent}
        a_1P_1 + a_2P_2 + \cdots + a_NP_N = O.
    \end{equation}
    Without loss of generality, assume $a_1 \neq 0$. Since 
    \[L_1 \cap L_2 \cdots L_N = K,\]
    we have 
    \[a_1P_1 = -(a_2P_2 + \cdots + a_NP_N) \in A(L_1) \cap A(L_2 \cdots L_N) = A(K),\]
    which is a contradiction.
\end{proof}

\begin{theorem}\label{Frey_Jarden_theorem}
    Let $A/K$ be an abelian variety over a number field $K$. Suppose there exists a smooth projective variety $X/K$ with a finite abelian subgroup $\Gamma \subseteq \Aut_K(X)$ such that $X/\Gamma \cong \bbp^d_K$, and there exists a non-constant morphism $f:X \rightarrow A$ over $K$. Then the rank of $A(K^{\mathrm{ab}})$ is infinite.
\end{theorem}
\begin{proof}
    By Corollary~\ref{main_theorem_corollary}, there exist infinitely many mutually linearly disjoint $\Gamma$-extensions $L_i$ over $K$ such that 
    \[\rank(A(L_i)) > \rank(A(K)).\]
    Let $L/K$ be the compositum of all extensions $L_i/K$ in $\overline{K}$. By Lemma~\ref{compositum_lemma}, the rank of $A(L)$ is infinite. Since $\Gamma$ is abelian, $L$ is contained in $K^{\mathrm{ab}}$, hence the rank of $A(K^{\mathrm{ab}})$ is infinite.
\end{proof}

\begin{rmk}
    The proof actually yields a stronger conclusion than Theorem~\ref{Frey_Jarden_theorem}. It proves that the rank of $A(L)$ is infinite, where $L$ is the compositum of all $\Gamma$-extensions of $K$.
\end{rmk}

\subsection{Larsen's conjecture}

Let $A/K$ be an abelian variety over a number field $K$. Frey and Jarden \cite{FJ74} proved that if $n$ is a positive integer, then for almost every tuple $(\sigma_1,\ldots,\sigma_n) \in G_K^n$, the rank of $A(\overline{K}^{\langle \sigma_1,\ldots,\sigma_n \rangle})$ is infinite, where $\overline{K}^{\langle \sigma_1,\ldots,\sigma_n \rangle}$ is the fixed field of $\overline{K}$ under $\sigma_1,\ldots,\sigma_n$. Motivated by this result, Larsen \cite[Question~1]{Lar03} asked whether the same conclusion holds for every tuple $(\sigma_1,\ldots,\sigma_n) \in G_K^n$.

\begin{conjecture}[Larsen]\label{Larsen}
    Let $A/K$ be an abelian variety over a number field $K$. Let~$G$ be a topologically finitely generated subgroup of $G_K := \Gal(\overline{K}/K)$. Then the rank of~$A(\overline{K}^G)$ is infinite.
\end{conjecture}

Several important cases of Conjecture~\ref{Larsen} are known. The second author and Larsen~\cite{IL08} proved Conjecture~\ref{Larsen} when $G$ is procyclic. Tim and Vladimir Dokchitser \cite{DD09} proved Conjecture~\ref{Larsen} for elliptic curves over $\bbq$, assuming either the rank part of the Birch and Swinnerton-Dyer conjecture or the finiteness of the Tate-Shafarevich group over arbitrary number fields. The second author and Larsen \cite{IL13} proved Conjecture~\ref{Larsen} for Jacobian varieties of split hyperelliptic curves. The first and second authors \cite{CI26} proved Conjecture~\ref{Larsen} for elliptic curves over $\bbq$ of analytic rank at most one. In a recent preprint, the second author and Larsen \cite{IL25} proved Conjecture~\ref{Larsen} for all elliptic curves. One can find more information about Conjecture~\ref{Larsen} in \cite{IL21}.

We now present our application of Theorem~\ref{main_theorem} to Conjecture~\ref{Larsen}. For the proof, we set one notation. For a finite group $\Delta$, we define $d(\Delta)$ to be the minimal number of generators of $\Delta$.

\begin{theorem}\label{Larsen_theorem}
    Let $A/K$ be an abelian variety over a number field $K$. Let $G$ be a topologically finitely generated subgroup of $G_K := \Gal(\overline{K}/K)$ with $n$ generators. Suppose there exists a smooth projective variety $X/K$ with a finite subgroup $\Gamma \subseteq \Aut_K(X)$ such that $X/\Gamma \cong \bbp^d_K$, and there exists a non-constant morphism $f_\Delta:X/\Delta \rightarrow A$ over $K$ for each subgroup $\Delta \leq \Gamma$ such that $d(\Delta) \leq n$. Then the rank of $A(\overline{K}^G)$ is infinite.
\end{theorem}
\begin{proof}
    Define 
    \[M_\Delta := \mathrm{Mor}_K(X/\Delta,A)/A(K)\]
    for each subgroup $\Delta \leq \Gamma$. By Corollary~\ref{main_theorem_fixed_field_corollary}, there exist infinitely many mutually linearly disjoint $\Gamma$-extensions $L_i$ over $K$ such that 
    \[\rank(A(L_i^\Delta)) \geq \rank(A(K)) + \rank(M_\Delta).\]
    We note that the image of $G$ in $\Gal(L_i/K) \cong \Gamma$ is generated by $n$ elements. Therefore, if we let $\Delta$ to be the image of $G$, then 
    \[\rank(A(L_i^G)) \geq \rank(A(K)) + \rank(M_\Delta) > \rank(A(K)).\]
    Let $L$ be the compositum of all extensions $L_i/K$ in $\overline{K}$. Then $L^G$ contains the compositum of all extensions $L_i^G/K$. By Lemma~\ref{compositum_lemma}, the rank of $A(L^G)$ is infinite. Therefore, the rank of $A(\overline{K}^G)$ is infinite.
\end{proof}

\begin{rmk}
    The proof actually yields a stronger conclusion than Theorem~\ref{Larsen_theorem}. It proves that the rank of $A(L^G)$ is infinite, where $L$ is the compositum of all $\Gamma$-extensions of $K$.
\end{rmk}

\end{document}